\documentclass[12 pt]{article}
\usepackage{amsmath, amsthm, amssymb,enumerate}
\usepackage{tikz}
\usepackage{fullpage}
\usepackage{enumitem}
\usepackage{hyperref}

\title{Beyond Ohba's Conjecture: A bound on the choice number of
$k$-chromatic graphs with $n$ vertices}
\author{
Jonathan A. Noel\thanks{Department of Mathematics and Statistics, McGill
University, Montr\'eal, QC, Canada, jonathan.noel@mail.mcgill.ca.
Current affiliation: Mathematical Institute, University of Oxford,
noel@maths.ox.ac.uk.  Research supported by an NSERC CGS M award.},
Douglas B. West\thanks{Departments of Mathematics, Zhejiang Normal University
(Jinhua, Zhejiang Province, China 321004) and University of Illinois
(Urbana, IL, USA 61801), west@math.uiuc.edu.  Research supported by Recruitment
Program of Foreign Experts, 1000 Talent Plan, State Administration of Foreign
Experts Affairs, China.},
Hehui Wu\thanks{Department of Mathematics, Simon Fraser University,
Burnaby, B.C., Canada V5A1S6, noshellwhh@gmail.com},
Xuding Zhu\thanks{Department of Mathematics, Zhejiang Normal University,
Jinhua, Zhejiang Province, China 321004, xdzhu@zjnu.edu.cn.
Research supported by NSF11171310 and ZJNSF Z6110786.}}

\newtheoremstyle{case}{}{}{\normalfont}{}{\itshape}{:}{ }{}

\newtheorem{thm}{Theorem}[section]
\newtheorem{lem}[thm]{Lemma}
\newtheorem{prop}[thm]{Proposition}
\newtheorem{conj}[thm]{Conjecture}
\newtheorem{cor}[thm]{Corollary}
\newtheorem{prob}[thm]{Problem}

\theoremstyle{definition}
\newtheorem{defn}[thm]{Definition}

\newtheorem{const}[thm]{Construction}
\newtheorem{step}{Step}
\newtheorem*{ack}{Acknowledgment:}

\newtheorem{rem}[thm]{Remark}

\newtheoremstyle{case}{}{}{\normalfont}{}{\itshape}{\normalfont:}{ }{}

\theoremstyle{case}
\newtheorem{case}{Case}

\newcommand{\ch}{\text{\rm ch}}

\def\CH#1#2{\binom{#1}{#2}}
\def\CL#1{\left\lceil{#1}\right\rceil}
\def\FL#1{\left\lfloor{#1}\right\rfloor}
\def\FR#1#2{\frac{#1}{#2}}
\def\nkup{\left\lceil\FR{n+k-1}3\right\rceil}
\def\Lunion{\bigl|\bigcup_{v\in V(G)}L(v)\bigr|}
\def\SE#1#2#3{\sum_{#1=#2}^{#3}}
\def\VEC#1#2#3{{#1_#2},\ldots,{#1}_{#3}}
\def\comment#1{}
\def\st{\colon\,}
\def\esub{\subseteq}
\def\nul{\varnothing}
\def\ZZ{{\mathbb Z}}

\numberwithin{equation}{section}

\begin{document}

\maketitle

\begin{abstract}
Let $\ch(G)$ denote the choice number of a graph $G$ (also called ``list
chromatic number'' or ``choosability'' of $G$).  Noel, Reed, and Wu proved the
conjecture of Ohba that $\ch(G)=\chi(G)$ when $|V(G)|\le 2\chi(G)+1$.  We
extend this to a general upper bound:
$\ch(G)\le \max\{\chi(G),\lceil({|V(G)|+\chi(G)-1})/{3}\rceil\}$.  Our result
is sharp for $|V(G)|\le 3\chi(G)$ using Ohba's examples, and it improves the
best-known upper bound for $\ch(K_{4,\dots,4})$.
\end{abstract}

\baselineskip 16pt

\section{Introduction}

Choosability is a variant of classical graph coloring; it models limited 
availability of resources.  Each vertex $v$ in a graph $G$ is assigned a list
$L(v)$ of available colors.  An \emph{$L$-coloring} is
a proper coloring $f$ of $G$ such that $f(v)\in L(v)$ for all $v\in V(G)$, and
$G$ is \emph{$k$-choosable} if $G$ has an $L$-coloring whenever $|L(v)|\geq k$
for all $v\in V(G)$.  The \emph{choice number} of $G$, denoted $\ch(G)$, is the
least $k$ such that $G$ is $k$-choosable.  Introduced by Vizing~\cite{Vizing}
and by Erd\H{o}s, Rubin, and Taylor~\cite{ERT}, choosability is now a
well-studied topic (surveyed in~\cite{Alon,Krat,Tuza,Woodall}).

Since $k$-choosability requires an $L$-coloring when $L(v)=\{1,\dots,k\}$
for all $v\in V(G)$, always $\ch(G)\ge\chi(G)$, where $\chi(G)$ is the chromatic
number.  However, there is no upper bound on $\ch(G)$ in terms of the chromatic
number $\chi(G)$ (even for bipartite graphs).  Such bounds exist when the
number of vertices is specified, since always $\ch(G)\le |V(G)|$, so it is
natural to seek the maximum of $\ch(G)$ among $k$-chromatic graphs with $n$
vertices.
Ohba~\cite{OhbaOrig} conjectured $\ch(G)=\chi(G)$ for $n\le 2\chi(G)+1$.
Several papers proved partial results in this direction
(see~\cite{Kostochka,OhbaOrig,ReedSudakov2,ReedSudakov,Alpha=3}),
and the conjecture has now been proved:

\begin{thm}[Noel, Reed, and Wu~\cite{NRW}]\label{NRW}
If $|V(G)|\leq 2\chi(G)+1$, then $\ch(G)=\chi(G)$.
\end{thm}

With $n$ and $k$ fixed, it suffices to study complete $k$-partite graphs,
since adding an edge cannot reduce the choice number.  In discussing complete
$k$-partite graphs, we use {\it part} rather than the formal term ``partite
set'' to mean a maximal stable set (also called ``independent set'', a
{\it stable set} is a set of pairwise nonadjacent vertices).  We write
$K_{1*k_1,2*k_2,\dots}$ for the complete multipartite graph with $k_i$ parts of
size $i$.

When $k$ is even, $K_{2*(k-1),4*1}$ and $K_{1*(k/2-1), 3*(k/2+1)}$ are not
$k$-choosable~\cite{EOOS}, making Theorem~\ref{NRW} sharp.  However, when $k$
is odd, $K_{2*(k-1),4*1}$ is $k$-choosable~\cite{EOOS}.  Since
$\ch(K_{2*(k-1),5*1})>k$ for all $k$~\cite{EOOS}, the only unsettled case is
whether $n=2k+2$ implies $\ch(G)=k$ when $k$ is odd.  Noel~\cite{Noel}
conjectured that the only complete $k$-partite graphs on $2k+2$ vertices that
are not $k$-choosable are $K_{1*(k/2-1),3*(k/2+1)}$ and $K_{2*(k-1),4*1}$ for
even $k$.

Moving to larger $n$, Ohba~\cite{3Ohba} determined the choice number
for a family of complete $k$-partite graphs with at most $3k$ vertices.

\begin{thm}[Ohba~\cite{3Ohba}]
\label{tight}
$\ch(K_{1*k_1,3*k_3})=
\max\left\{k,\left\lceil\frac{n+k-1}{3}\right\rceil\right\}$,
where $k=k_1+k_3$ and $n=k_1+3k_3$.
\end{thm}

Earlier, Kierstead~\cite{Kierstead} computed the special case
$\ch(K_{3*k})= \CL{\FR{4k-1}3}=\CL{\FR{n+k-1}3}$.  Our main result relies
on Theorem~\ref{NRW} and extends Ohba's upper bound to all graphs.  It is sharp
when $n\le 3k$ (with $n-k$ even) and useful when $n$ is bounded by a small
multiple of $k$, but it is weak when $n/k$ is large:

\begin{thm}
\label{main}
For any graph $G$ with $n$ vertices and chromatic number $k$,
\[\ch(G)\leq \max\left\{k,\left\lceil\frac{n+k-1}{3}\right\rceil\right\}.\]
\end{thm}


%

For use in bounding $\ch(G)$ for the random graph,~\cite{ERT} suggested finding
good bounds on $\ch(K_{m*k})$.  By our result, $K_{3*k}$ has the largest choice
number among $k$-chromatic graphs with at most $3k$ vertices.  For $m=4$,
Yang~\cite{Daqing} proved $\FL{\FR{3k}2}\le\ch(K_{4*k})\le\CL{\FR{7k}4}$; our
result improves the upper bound to $\CL{\frac{5k-1}3}$.  Since the writing of
our paper, Kierstead, Salmon, and Wang~\cite{KSW} have determined that
$\ch(K_{4*k})$ equals the easy lower bound $\FL{\FR{3k}2}$.  Our result also
yields $\FL{\FR{8k}5}\le \ch(K_{5*k})\le 2k$ and 
$\FL{\FR{5k}3}\le \ch(K_{6*k})\le \CL{\FR{7k-1}3}$.  The bounds
$\FL{\FR{2k(m-1)}m}\le \ch(K_{m*k})\le \CL{\FR{k(m+1)-1}3}$ are valid for all
$m$, but they are interesting only for small $m$.  The lower bound arises from
the following construction.

\begin{const}\label{smallm}
Consider a universe $U$ of $2k-1$ colors, split into sets $\VEC X1m$ of sizes
$\FL{\FR{2k-1}m}$ and $\CL{\FR{2k-1}m}$.  Assign list $U-X_i$ to the $i$th
vertex of each part.  An $L$-coloring must use at least two colors on
each part, and these pairs must be disjoint.  Hence at least $2k$ colors
must be used, but $|U|<2k$, so there is no $L$-coloring.  The list sizes are
all at least $2k-1-\CL{\FR{2k-1}m}$, which is at least $\FL{\FR{2k(m-1)}m}-1$.
(For sharpness of Theorem~\ref{main} when $n=3k-2i$, use $2k-1-i$ colors, with
$k-i$ parts of size $3$ and $i$ singleton parts whose list is the full color
set.)
\end{const}

\begin{prob}
Determine $\ch(K_{m*k})$ for small $m$, beginning with $m=5$.
\end{prob}

Noel~\cite{Noel} conjectured that in fact $K_{m*k}$ always has the largest
choice number among $k$-chromatic graphs with at most $mk$ vertices.  More
generally:

\begin{conj}[Noel~\cite{Noel}] \label{n/k}
For $n\geq k\geq 2$, among $n$-vertex $k$-chromatic graphs the choice number is
maximized by a complete $k$-partite graph with independence number $\CL{n/k}$.
\end{conj}

Theorem~\ref{NRW} implies Conjecture~\ref{n/k} for $n\leq 2k+1$.  
Theorems~\ref{main} and~\ref{tight} together imply Conjecture~\ref{n/k} when
$n\leq 3k$ and $n-k$ is even.

Construction~\ref{smallm} never yields lists of size exceeding $2k$.  When $m$
is large in terms of $k$, a simple explicit construction generalizing an
example in \cite{ERT} gives a good lower bound.

\begin{const}\label{largem}
Let $m=\CH{kj-1}{(k-1)j}$.  For $K_{m*k}$, assign all $(k-1)j$-sets from a set
$U$ of $kj-1$ colors as lists on each part.  Each list omits $j-1$ colors.
If fewer than $j$ colors are used on some part, then this part has a list from
which no color is chosen.  Hence $kj$ colors are needed, but $|U|<kj$.
When $k$ is large, the leading behavior of $m$ is a constant times
$k^j/\sqrt j$, which yields $j\approx \log_k m+\FR12 \log_k\log_k m$.  Thus
$\ch(K_{m*k})\ge c(k-1)\FR{\log m}{\log k}$.
\end{const}

For intermediate $m$, Constructions~\ref{smallm} and~\ref{largem} can be
combined.  Alon~\cite{Alonlog} improved the latter, proving
$c_1k\log m\le\ch(K_{m*k})\le c_2k\log m$ for some constants $c_1$ and
$c_2$.  This yields choice number $O(\FR{n\log\log n}{\log n})$ for the random
graph.  More precise asymptotic bounds were later obtained by Gazit and
Krivelevich~\cite{Kriv}: $\ch(K_{m*k})=(1+o(1))\FR{\log m}{\log(1+1/k)}$.
With Conjecture~\ref{n/k}, the general upper bound would be
$\ch(G)\le (1+o(1))\FR{\log(n/k)}{\log(1+1/k)}$.
\bigskip

Our proof of Theorem~\ref{main} begins with several restrictions on minimal
counterexamples.  First, Theorem~\ref{NRW} verifies the claim in the most
difficult range ($n\le 2k+1$), which will serve as a basis.  In that range
the lists have size only $k$; when the problem is restricted to $n\ge 2k+2$,
the lists will always have size at least $k+1$.  As noted, we may assume that
$G$ is a complete $k$-partite graph.  We prove that in a minimal
counterexample, all parts have size at most $4$ and no color appears in more
than two lists on one part.  We then produce an $L$-coloring, contradicting the
assumption of a counterexample.

\begin{step}
\label{S1}
Break $V(G)$ into stable sets of size at most $2$ by splitting some parts.
\end{step}

\begin{step}
\label{S2}
Produce an $L$-coloring whose color classes are the sets obtained in
Step~\ref{S1}.
\end{step}

In Section~\ref{outline}, we prove that if a partition of the type in
Step~\ref{S1} satisfies several special properties, then Hall's
Theorem~\cite{Hall} on matchings in bipartite graphs produces an $L$-coloring
to complete Step~\ref{S2}.  In Sections~\ref{greed} and~\ref{additional}, we
show that $V(G)$ admits a partition satisfying these special properties,
thereby completing Step 1 and the proof of Theorem~\ref{main}.

\section{Preliminary Reductions}\label{prelim}

If Theorem~\ref{main} is not true, then there is a minimal counterexample.

\begin{rem}\label{mincex}
If Theorem~\ref{main} fails, then by Theorem~\ref{NRW} there is an $n$-vertex
complete $k$-partite graph $G$ with list assignment $L$ such that $G$ has no
$L$-coloring, $n\ge2k+2$, $|L(v)|\ge\nkup\ge k+1$ for all $v\in V(G)$, and
the conclusion of Theorem~\ref{main} holds for all graphs with fewer vertices.
\end{rem}

Henceforth, $G$ and $L$ will have the properties stated in Remark~\ref{mincex}.
We will derive additional properties, after which we will produce an
$L$-coloring of $G$.  For example, we may assume $\Lunion<n$ due to the
following lemma proved independently by Kierstead~\cite{Kierstead} and by Reed
and Sudakov~\cite{ReedSudakov,ReedSudakov2}.

\begin{lem}[\cite{Kierstead,ReedSudakov,ReedSudakov2}] \label{<n}
If $G$ is not $r$-choosable, then there is a list assignment $L$ such
that $G$ has no $L$-coloring, all lists have size at least $r$, and their
union has size less than $|V(G)|$.
\end{lem}

The reduction $\Lunion<n$ is a standard reduction for minimal counterexamples
in choosability problems, so much so that it has a name: the ``Small Pot
Lemma''.  It has been applied in diverse situations, including
\cite{Choi,CR1,CR2,CR3,NRW,KSW}.

Next, we obtain more specific restrictions on $G$
and $L$ for our problem.  These use the following key proposition.

\begin{prop}
\label{min}
If $A$ is a stable set in $G$ whose lists have a common color, then 
\[\left\lceil\frac{|V(G-A)|+\chi(G-A)-1}{3}\right\rceil=\left\lceil\frac{n+k-1}{3}\right\rceil.\]
\end{prop}

\begin{proof}
Let $c\in\bigcap_{v\in A}L(v)$.  Let $G'=G-A$, and let $L'$ be the list
assignment for $G'$ obtained from $L$ by deleting $c$ from each list containing
it.

Since $|L(v)|\ge k+1$, we have $|L'(v)|\ge k\ge\chi(G')$ for all $v\in V(G')$.
Also $|L'(v)|\ge \nkup-1$.
If $\CL{\FR{|V(G')|+\chi(G')-1}3}<\nkup$, then
$|L'(v)|\ge\max\left\{\chi(G'),\CL{\FR{|V(G')|+\chi(G')-1}3}\right\}$.  By the 
minimality of $G$, we obtain an $L'$-coloring of $G'$, which extends to an
$L$-coloring of $G$ by giving color $c$ to $A$.  Hence
$\CL{\FR{|V(G')|+\chi(G')-1}3}\ge\nkup$, and $G'\esub G$ yields equality.
\end{proof}

\begin{cor} \label{no2}
The lists on a part of size $2$ in $G$ are disjoint.
\end{cor}
\begin{proof}
A shared color in a part $A$ of size $2$ contradicts Proposition~\ref{min} via 
$\CL{\FR{V(G-A)+\chi(G-A)-1}3}=\CL{\FR{(n-2)+(k-1)-1}3}<\nkup$.
\end{proof}

\begin{cor}\label{no3}
Each color appears in at most two lists in each part in $G$.
\end{cor}
\begin{proof}
Having three vertices with a common color in a part $A$  contradicts
Proposition~\ref{min} via
$\CL{\FR{|V(G-A)|+\chi(G-A)-1}3}\le \CL{\FR{(n-3)+k-1}3}<\nkup$.
\end{proof}

\begin{lem}\label{alpha}
$\alpha(G)\leq 4$. 
\end{lem}

\begin{proof}
Let $A$ be a stable set in $G$. By Lemma~\ref{no3}, each color appears in
at most two lists on $A$, so $\sum_{v\in A}|L(v)|\le 2\Lunion\le 2(n-1)$,
by Lemma~\ref{<n}.  Also, $\sum_{v\in A}|L(v)|\ge |A|\nkup$.  Together, the
inequalities yield $|A|\le 6\FR{n-1}{n+k-1}$, so $|A|\le 5$.  If equality
holds, then $n\ge 5k+1$, which requires a part of size at least $6$ and is
already forbidden.
\end{proof}

The restrictions so far simplify the main approach.  The remaining restrictions
in this section are technical statements used to simplify the arguments in
Sections~\ref{greed} and~\ref{additional} that $V(G)$ admits a partition
satisfying the properties specified in Section~\ref{outline}.  We obtain them
here as further consequences of Proposition~\ref{min}.

\begin{lem}\label{nk13}
$\FR{n+k-1}3$ is an integer.
\end{lem}

\begin{proof}
Let $A$ be a largest part, so $n\leq k|A|$.  If the lists on $A$ are
disjoint, then 
\[n\leq k|A|\leq\sum_{v\in A}|L(v)|
=\Big|\bigcup_{v\in A}L(v)\Big|\leq\Big|\bigcup_{v\in V(G)}L(v)\Big|<n.\]
Hence $A$ contains a $2$-set $A'$ with intersecting lists, which by
Corollary~\ref{no3} and $n>2k$ is not all of $A$.  Now
\[\CL{\FR{|V(G-A')|+\chi(G-A')-1}3}\le\CL{\FR{(n-2)+k-1}3}=\FL{\FR{n+k-1}3}.\]
If $\FR{n+k-1}3\notin\ZZ$, then this contradicts Proposition~\ref{min}.
%
\end{proof}

Henceforth let $k_i$ be the number of parts with size $i$, for
$i\in\{1,2,3,4\}$.  Note that $k=\SE i14 k_i$ and $n=\SE i14 ik_i$.

\begin{cor}\label{arith}
The parameters $k_1,k_2,k_3,k_4$ satisfy the following relationships.\\
(a) $\FR{n+k-1}3=k+k_4-\FR{k_1-k_3+k_4+1}3$, with both fractions being integers.\\
(b) $\FR{n+k-1}3+\FR k3\ge k+k_4+\FR{2k_3-1}3$.\\
(c) $\FR{2(n+k-1)}3=n+\FR{k_1-k_3-2k_4-2}3=k+k_3+2k_4+\FR{k+2k_2+k_3-2}3$.
\end{cor}
\begin{proof}
(a) $n+k=2k_1+3k_2+4k_3+5k_4=3k-k_1+k_3+2k_4$.  Integrality was shown in
Lemma~\ref{nk13}.

(b) Use $k\ge k_1+k_3+k_4$ in the right side of (a).

(c) $2(n+k)=4k_1+6k_2+8k_3+10k_4 = 4k+2k_2+4k_3+6k_4$.
\end{proof}

\section{A Sufficient Condition for $L$-Coloring} \label{outline}
By Corollary~\ref{no3}, each color appears in at most two lists in each part.
Therefore, an $L$-coloring must refine the partition of $V(G)$ into stable sets
of size at most $2$.  To find an $L$-coloring, we must determine which pairs
will form the color classes of size $2$.

\begin{defn}
\emph{Merging} non-adjacent vertices $u$ and $v$ in $G$ means replacing $u$ and
$v$ by one vertex $w$ with list $L(w)$ equal to $L(u)\cap L(v)$.  We then say
that $w$ is a \emph{merged} vertex.  Given a set $S$ of vertices, each of which
may be merged or unmerged, let $L(S)=\bigcup_{v\in S}L(v)$.
\end{defn}

A \emph{system of distinct representatives (SDR)} for a family
$\{X_1,\dots,X_m\}$ of sets is a set $\{x_1,\dots,x_m\}$ of distinct elements
such that $x_i\in X_i$ for $i\in\{1,\dots,m\}$.  Our goal is to perform some
merges so that the resulting color lists have an SDR; assigning each vertex the
chosen representative of its list (with the color chosen for a merged vertex
used on both original vertex comprising it) then yields an $L$-coloring of $G$.
To facilitate finding an SDR, it is natural to merge vertices whose lists 
have many common colors.  Merged vertices come from the same part in $G$ and
will only merge two previously unmerged vertices, since by Corollary~\ref{no3}
a merge of three vertices would have an empty list of colors.

In this section we prove that if the merging procedure satisfies the properties
in Definition~\ref{rich}, then the desired SDR exists.  The proof uses
Hall's Theorem~\cite{Hall}, which states that $\{X_1,\dots,X_m\}$ has an SDR if
and only if $\big|\bigcup_{i\in S}X_i\big|\ge|S|$ for all $S\esub\{1,\dots,m\}$.
The proof of Theorem~\ref{main} will be completed in Sections~\ref{greed}
and~\ref{additional} by showing that merges can be performed to establish
these properties.

\begin{defn}\label{rich}
Let a \emph{$j$-part} be a part having size $j$ in $G$ (before merges).
After performing merges, let $A^*$ denote the set of vertices resulting from
the part $A$.  Let $t_3$ denote the number of $3$-parts having merged vertices.
Let $Z_3$ be some fixed set of $\FL{\FR23 k_3}$ $3$-parts, and let $Z_4$ be
some fixed set of $\max\{0,\FR{k_1-k_3+k_4+1}3\}$ $4$-parts.  For a set of
merges, we define properties (P1)--(P8) below.  Note that (P3)-(P7) can be
considered for individual parts.

\smallskip
(P1) $t_3\ge{k_3/3}$.

(P2) In every $4$-part, at least one merge occurs.

(P3) If $x,y,z\in A^*$ are distinct, then 
$|L(x)\cup L(y)\cup L(z)|\ge n-t_3-k_4$.

(P4) If $|A^*|=|A|=3$ and $x,y\in A^*$, then
$|L(x)\cup L(y)|\ge k+k_3+k_4$.

(P5) If $A\in Z_3$ and $x,y\in A^*$, then
$|L(x)\cup L(y)|\ge k+t_3+k_4$.

(P6) If $|A|=3$ and $x,y\in A^*$, then
$|L(x)\cup L(y)|\ge k+{\FR{k_3}3}+k_4$.

(P7) If $A\in Z_4$ and $x,y\in A^*$, then
$|L(x)\cup L(y)|\ge k+k_4$.

(P8) The set of lists of merged vertices has an SDR.
\end{defn}

In the specification of $Z_4$, note that $\FR{k_1-k_3+k_4+1}3$ is an integer, by
Corollary~\ref{arith}(a).  Property (P3) applies to $3$-parts without merges
and to $4$-parts with one merge.  To understanding the intuition behind using
Hall's Theorem to show that these properties are sufficient, note that the
lower bounds in (P3)--(P7) are successively weaker (the comparison of the
bounds in (P5) and (P6) uses (P1)).  A large set $S$ must contain vertices
whose lists are large, thereby satisfying $|L(S)|\ge |S|$ and allowing such
sets to be excluded.  As the remaining sets to be considered become smaller
by eliminating such vertices, smaller lower bounds on the list sizes become
sufficient.  Property (P8) can then be viewed as reducing the problem to
finding an SDR of a smaller family (generated by the merged vertices) when
(P1)--(P7) hold.

\begin{lem}\label{P1P8}
When the merges satisfy (P1)-(P8), the family of all resulting lists has an
SDR.
\looseness-1
\end{lem}

\begin{proof}
By Hall's Theorem, it suffices to prove $|L(S)|\ge|S|$ for each vertex set $S$
(after the merges).  We use (P1)--(P7) to restrict $S$ until it consists only
of merged vertices, and then (P8) guarantees $|L(S)|\ge|S|$ for such $S$.

By (P2), the merges leave at most $n-t_3-k_4$ vertices, so $|S|\le n-t_3-k_4$.
Thus (P3) yields $|L(S)|\ge|S|$ whenever $S$ has three vertices from one part.
We may thus restrict $S$ to having at most two vertices from each part,
yielding $|S|\le k+k_2+k_3+k_4\le 2k$.

If $S$ contains both vertices from a part of size $2$ (unmerged), then by
Corollary~\ref{no2} their lists are disjoint and $|L(S)|\ge 2k+2>|S|$;
hence $|S|\le k+k_3+k_4$.  If $S$ contains two vertices from a $3$-part with no
merged vertices, then (P4) yields $|L(S)|\ge k+k_3+k_4$; hence
$|S|\le k+t_3+k_4$.  If $S$ contains two vertices from a $3$-part in $Z_3$,
then (P5) yields $|L(S)|\ge k+t_3+k_4$; hence $|S|\le k+\CL{\FR{k_3}3}+k_4$,
since $\CL{\FR{k_3}3}=k_3-|Z_3|$.  If $S$ contains two vertices from any
$3$-part, then (P6) yields $|L(S)|\ge k+\CL{\FR{k_3}3}+k_4$; hence
$|S|\le k+k_4$.  If $S$ contains two vertices from a $4$-part in $Z_4$, then 
(P7) yields $|L(S)|\ge k+k_4$; hence $|S|\le k+k_4-|Z_4|\le\FR{n+k-1}3$, by
Corollary~\ref{arith}(a) and the formula for $|Z_4|$.  Now $|L(S)|\ge|S|$ if
$S$ contains any unmerged vertex.  Finally, (P8) applies.
\end{proof}

In the rest of the paper, we describe an explicit procedure to obtain such
merges.

\section{Greedy Merges} \label{greed} 
In order to guarantee (P3)--(P8), it is helpful to merge vertices whose lists
have large intersection.  Our first task is to make the meaning of ``large''
precise.

\begin{defn}\label{defgood}
For a part $A$ in $G$, let
$\ell(A)=\max\left\{|L(u)\cap L(v)|\st u,v\in A\right\}$.\\
If $|A|\ge3$, then a pair $\{u,v\}\esub A$ is a \emph{good pair} for $A$ if

$|A|=3$ and $|L(u)\cap L(v)|\geq \frac{k_1+k_4+1}{3}$,
or if

$|A|=4$ and $|L(u)\cap L(v)|\geq |L(w)\cap L(z)|$, where $\{w,z\}=A-\{u,v\}$.

\noindent
The merge of a good pair is a \emph{good merge}.  A part $A$ is \emph{good} if
a good merge is made in it.
\end{defn}

When $|A|=4$, a pair in $A$ whose lists have largest intersection is good by
definition.  With a lower bound on $\ell(A)$, this will also hold when $|A|=3$.

\begin{lem}
\label{geqk}
If $T\subseteq V(G)$ is a stable set of size $3$, then
$\sum_{\{u,v\}\in \CH T2}|L(u)\cap L(v)|\geq k$.
\end{lem}

\begin{proof}
Since colors appear at most twice in each part, and $|L(V(G))| <n$, we have
\[\sum_{\{u,v\}\in \CH T2}|L(u)\cap L(v)| = \sum_{u\in T}|L(u)|-|L(T)|
\ge 3\left(\frac{n+k-1}{3}\right) - (n-1) = k.\]

\vspace{-3.3pc}
\end{proof}
\bigskip

\begin{cor} \label{k/3}
If $|A|\geq 3$, then $\ell(A)\geq\frac{k}{3}$. 
\end{cor}

\begin{cor} \label{ellgood}
When $|A|\geq 3$, a pair $\{u,v\}\subseteq A$ maximizing $|L(u)\cap L(v)|$ is good for $A$.
\end{cor}

\begin{proof}
When $|A|=4$, the conclusion is immediate from Definition~\ref{defgood}.  If
$|A|=3$, then $k_3\geq 1$.  Thus
$|L(u)\cap L(v)|=\ell(A)\geq\FR k3 \geq\frac{k_1+k_4+1}3$,
so $\{u,v\}$ is good for $A$.
\end{proof}

%

\begin{lem} \label{reduce3}
Every good $3$-part $A$ satisfies (P6).
\end{lem}
\begin{proof}
We have $A^*=\{x,y\}$ and may assume that $y$ is merged and $x$ is not.  By
Corollary~\ref{no3}, $L(x)\cap L(y)=\nul$, and forming $y$ by a good merge
yields
\[|L(x)\cup L(y)| = |L(x)|+|L(y)|
\geq\frac{n+k-1}{3}+{\FR{k_1+k_4+1}3} = k+{\frac{k_3}3}+k_4,\]
by Corollary~\ref{arith}(a).  Thus, the desired inequality holds.
\end{proof}

\begin{lem}\label{reduce4}
If (P1) holds and $A$ is a good $4$-part, then $A$ satisfies (P3).
\end{lem}

\begin{proof}
Let $A$ be a $4$-part, with $x,y,z\in A^*$.  Since $A$ is good, we may assume
that $x$ is merged and that $y$ and $z$ are not.  Since $x$ arises from a good
merge, $|L(x)|\geq |L(y)\cap L(z)|$.  Also,
$L(x)\cap L(y)=L(x)\cap L(z)=\nul$ by Corollary~\ref{no3}. Therefore,
\begin{align*}
|L(x)\cup L(y)\cup L(z)|&=|L(x)|+|L(y)|+|L(z)|-|L(y)\cap L(z)|\\
&\geq |L(y)|+|L(z)|\geq \frac{2(n+k-1)}{3}\geq n+\FR{k_1-k_3-2k_4-2}3,
\end{align*}
by Corollary~\ref{arith}(c).  By (P1), we have $t_3\geq k_3/3$, and clearly
$k_1\ge 0$ and $2k_4/3<k_4$, so
$|L(y)|+|L(z)|\ge n+{\FR{k_1-k_3-2k_4-2}3}\ge n-t_3-k_4$, and the
desired conclusion holds.
\end{proof}

Finally, we are ready to specify merges.  We will specify merges in special
sets $Z_3$ of $3$-parts and $Z_4$ of $4$-parts.  These will make each such part
good, though when we make two merges in a $4$-part in $Z_4$ they need not both
be good.

We will do this in 
subsequently specified in each $3$-part or $4$-part outside $Z_3\cup Z_4$, then
the set of merges will satisfy properties (P1)--(P7) and Q1.  In
Section~\ref{additional}, we will show that those remaining good merges
can then be chosen so that (P8) is also satisfied.

\subsection{Parts of Size 3}

Specify a fixed set $Z_3$ of exactly $\FL{\FR{2k_3}3}$ $3$-parts; exactly
$\CL{\FR{k_3}3}$ $3$-parts lie outside $Z_3$.  For every $3$-part outside $Z_3$,
we will choose a good merge in the next section.  Here we choose merges in some
members of $Z_3$ based on intersection sizes; only good pairs will be merged.
Note that we have not yet specified the value $t_3$ giving the number of
$3$-parts that will have merges.

\begin{const}\label{z3}
Set $t_3$ to be the largest integer for which there exists a set $Z_3'\esub Z_3$
of size $t_3-\CL{\frac{k_3}{3}}$ such that $\ell(A)\geq\frac{k+t_3-1}{3}$ for
all $A\in Z_3'$.  For $A\in Z_3$, merge a pair in $A$ achieving $\ell(A)$ if
and only if $A\in Z_3'$.
\end{const}

By Corollary~\ref{ellgood}, a pair whose lists have largest intersection in
a part is always a good pair.  Possibly no member of $Z_3$ has such a large
intersection size, in which case $t_3=\CL{\FR{k_3}3}$ and $Z_3'$ is empty.
In any case, $t_3\ge k/3$.
%
%

\begin{lem}
\label{3verts3}
If one merge is made in each $3$-part outside $Z_3$, then (P1) holds, (P3)
holds for $3$-parts, (P4) and (P5) hold, and (P6) holds for $3$-parts
in $Z_3$.
\end{lem}

\begin{proof}
If we later merge one pair in each $3$-part outside $Z_3$ (there are
$\CL{\FR{k_3}3}$ such parts), then the total number of merges in $3$-parts 
will be $t_3$, and (P1) holds.

For $3$-parts, (P3) and (P4) apply only to those without merges, all lying in
$Z_3-Z_3'$.  Membership in $Z_3-Z_3'$ requires $\ell(A)\le {\frac{k+t_3-2}3}$.
By Corollary~\ref{no3},
\begin{align*}
|L(A)|&= \sum_{v\in A}|L(v)| - \sum_{\{u,v\}\in \CH A2} |L(u)\cap L(v)|\\
&\ge \FR{3(n+k-1)}3 - \FR{3(k+t_3-2)}3 > n-t_3\geq n-t_3-k_4,
\end{align*}
which proves (P3).  For (P4), we take just two vertices $x,y\in A$.  We compute
\begin{align*}
|L(x)\cup L(y)|&=|L(x)|+|L(y)|-|L(x)\cap L(y)|\\
&\ge \FR{2(n+k-1)}3-\FR{k+t_3-2}3\ge k+k_3+2k_4+\FR{2k_2}3\ge k+k_3+k_4
\end{align*}
using Corollary~\ref{arith}(c) and $k_3\ge t_3$.

Since $k_3\ge t_3$ and $k_3\ge k_3/3$, (P4) implies (P5) and (P6) for $3$-parts
containing no merge.  Since (P5) is imposed only for parts in $Z_3$, it
therefore suffices to consider $A\in Z_3'$.  By Corollary~\ref{ellgood}, the
merge in $A$ is good, so by Lemma~\ref{reduce3} $A$ satisfies (P6).  For (P5),
we may assume that $y$ is merged and $x$ is not.  Since $L(x)\cap L(y)=\nul$ by
Corollary~\ref{no3}, Construction~\ref{z3} yields
\begin{align*}
|L(x)\cup L(y)|&=|L(x)|+|L(y)|\geq \frac{n+k-1}3+{\frac{k+t_3-1}3} \\
&\geq k+\frac{k_2+2k_3+t_3}{3} + k_4-\frac{2}{3}\geq k+t_3+k_4-\frac{2}{3},
\end{align*}
using Corollary~\ref{arith}(b) and $k_3\ge t_3$.  Hence
$|L(x)\cap L(y)|\geq k+t_3+k_4$, as desired.
\end{proof}

To complete the proof of (P6) for all parts, it suffices by Lemma~\ref{reduce3}
to specify a good merge in each $3$-part outside $Z_3$.

\subsection{Parts of Size 4}
Corollary~\ref{arith}(a) and $n\ge2k+2$ imply that $\FR{k_1-k_3+k_4+1}3$ is an
integer less than $k_4$.  Hence we can specify a fixed set $Z_4$ consisting of
exactly $\max\left\{0,\frac{k_1-k_3+k_4+1}{3}\right\}$ of the $4$-parts, chosen
arbitrarily.  In the next section we will choose one good pair to merge in each
$4$-part not in $Z_4$.  Here we specify one or two merges in each member of
$Z_4$, based on intersection sizes.

\begin{const}\label{merge4}
For $A\in Z_4$, merge a pair $\{u,v\}$ such that $|L(u)\cap L(v)|=\ell(A)$.
Also merge the remaining pair $\{w,z\}$ if
$|L(w)\cap L(z)|\geq s$, where $s=\frac{2n-k+1}{3}-k_4$.
\end{const}

\begin{rem} \label{z4rem}
Lemma~\ref{nk13} implies that $s$ is an integer.  Also $s\ge \FR k3$, since
$s-\FR k3>\FR{2(n-k)}3-k_4\ge\FR{6k_4}3-k_4$.  Since also
$\ell(A)\ge{\FR k3}$ (Corollary~\ref{k/3}), the list of any merged vertex in
a member of $Z_4$ has size at least ${\FR k3}$.
\end{rem}

\begin{lem}
The merging procedure satisfies (P7).
\end{lem}

\begin{proof}
Property (P7) applies only for $A\in Z_4$.  Given $A\in Z_4$ and $x,y\in A^*$,
we need $|L(x)\cup L(y)|\ge k+k_4$.  We consider three cases, depending on the
merges in $A$.

\begin{case}
{\it Neither $x$ nor $y$ is merged.}
By Construction~\ref{merge4},
$|L(x)\cap L(y)|\leq s-1$.  Thus
\begin{align*}
|L(x)\cup L(y)|&\geq |L(x)|+|L(y)|-|L(x)\cap L(y)| 
\ge\FR{2(n+k-1)}3-\FR{2n-k-2}3+k_4=k+k_4.
\end{align*}
\end{case}

\begin{case}
{\it Exactly one of $\{x,y\}$, say $y$, is merged.}
By Corollary~\ref{no3}, $L(x)\cap L(y)=\nul$.  Thus, by Remark~\ref{z4rem} and
Corollary~\ref{arith}(b),
\[|L(x)\cup L(y)|\geq |L(x)|+|L(y)| \geq
\frac{n+k-1}{3}+\FR k3 \geq k+k_4-\frac{1}{3},\]
and $|L(x)\cup L(y)|$ is an integer.
\end{case}

\begin{case}
{\it Both $x$ and $y$ are merged.}
By Construction~\ref{merge4} and symmetry, we may assume
$\ell(A)=|L(x)|\ge|L(y)|\ge s$.  Since $L(x)\cap L(y)=\nul$ by
Corollary~\ref{no3}, we compute
\begin{align*}
|L(x)\cup L(y)|&=|L(x)|+|L(y)|\ge2s =n+\FR{n-2k+2}3-2k_4>
\left(\sum ik_i\right)-2k_4\ge k+k_4.
\end{align*}

In each case, the desired inequality holds.
\end{case}

\vspace{-2pc}
\end{proof}

Note that the verification of (P7) did not use any property of merges in
parts outside $Z_4$.

\section{The Remaining Merges} \label{additional}

At this point, we can reduce the proof of Theorem~\ref{main} to one main task.

\begin{lem}\label{reduction}
If in addition to the merges previously specified, it is possible to specify
one good merge in each $3$-part outside $Z_3$ and each $4$-part outside $Z_4$
in such a way that (P8) holds, then Theorem~\ref{main} is true.
\end{lem}
\begin{proof}
Specifying any merge in each $3$-part outside $Z_3$ completes the proofs of
(P1), (P4), and (P5), and it completes the proof of (P3) for $3$-parts
(by Lemma~\ref{3verts3}).  If those merges are good, then by
Lemma~\ref{reduce3} we also have proved (P6) completely.

Furthermore, we have previously proved (P7) completely, since it applies only
to parts in $Z_4$.  Specifying any merge in each $4$-part outside $Z_4$
completes the proof of (P2).  If those merges are good and we have specified
merges outside $Z_3$, then by Lemma~\ref{reduce4} we have (P3) also for
$4$-parts.

If the merges also satisfy (P8), then Lemma~\ref{P1P8} completes the proof.
\end{proof}

Hence our task is to merge a good pair in every $3$-part outside $Z_3$ and
every $4$-part outside $Z_4$ in such a way that (P8) holds.  In fact, we
specify the merges among the possible good merges {\it by requiring that (P8)
holds}.  Note that since we have already proved (P7) completely, we can use it
in this section.

Let $T$ denote the set of all merged vertices in parts in $Z_3\cup Z_4$.
Let $Y$ denote the set of parts of size $3$ or $4$ outside $Z_3\cup Z_4$.
To complete the proof, we need to find distinct colors, one for each vertex of
$T$ and one for each part in $Y$, such that the color chosen for each $A\in Y$
belongs to both lists for a good pair in $A$, and the color chosen for a merged
vertex $w$ in $T$ belongs to $L(w)$.  To obtain such a set of colors (and
thereby define the remaining merges), we again apply Hall's Theorem.

\begin{defn}
For $A\in Y$, let $L_A$ be the set of all colors $c$ such that
$c\in L(u)\cap L(v)$ for some good pair $\{u,v\}\esub A$.  Let $X$ be the
family of sets consisting of $L_A$ for all $A\in Y$ and $L(w)$ for all $w\in T$.
\end{defn}

We seek an SDR for $X$.  We start with lower bounds on $|L_A|$ for $A\in Y$. 
Note that this special list $L_A$ differs from $L(A)$, which we defined
to be $\bigcup_{v\in A}L(v)$.

\begin{lem}
\label{vp3}
If $A\in Y$ and $|A|=3$, then $|L_A|\geq k_3+{\FR{k_1+k_4}3}$.
\end{lem}

\begin{proof}
By Corollary~\ref{ellgood}, some pair in $A$ is good.  If $\{u,v\}$ is not
good, then $|L(u)\cap L(v)|\leq {\FR{k_1+k_4}3}$, by Definition~\ref{defgood}.
At most two pairs are not good, so Lemma~\ref{geqk} and $k=\sum k_i$ yield
\[|L_A|\ge k-\FR{2(k_1+k_4)}3 \geq k_3+\frac{k_1+k_4}{3}.\]

\vspace{-2pc}
\end{proof}
\medskip

\begin{lem}
\label{vp4}
If $A\in Y$ and $|A|=4$, then $|L_A|\geq k_3+k_4$.
\end{lem}

\begin{proof}
By Corollary~\ref{no3} and the definition of a good pair for $A$,
\[|L_A|\ge\FR12\sum_{\{u,v\}\in \CH A2}|L(u)\cap L(v)| =
\FR{\sum_{u\in A}|L(u)|}2 - \FR{\bigl|\bigcup_{u\in A}L(u)\bigr|}2.\]
Since the union of all lists has fewer than $n$ colors,
\[|L_A|\ge\FR12\left(\FR{4(n+k-1)}3 -(n-1)\right) =\FR{n+4k-1}6>k\ge k_3+k_4,\]
since $n>2k+1$.
\end{proof}

As argued in Lemma~\ref{reduction}, the following lemma completes the
proof of Theorem~\ref{main}.

\begin{lem}
There is an SDR for $X$. 
\end{lem}

\begin{proof}
As in Lemma~\ref{P1P8}, we check Hall's Condition by verifying
for successively restricted $S\esub X$ that the union of the lists indexed
by $S$ has size at least $|S|$.  By construction, each $3$-part contributes at
most one list to $X$, each $4$-part outside $Z_4$ also contributes at most one,
and each part in $Z_4$ contributes at most two.
Hence $|S|\le |X|\le k_3+k_4+|Z_4|$.

If $S$ contains two lists for a part $A$, then $A\in Z_4$.  By (P7), the union
of these two lists has size at least $k+k_4$, which exceeds $|X|$.  Hence $S$
has at most one list from each part, so $|S|\le k_3+k_4$.  By Lemma~\ref{vp4},
we are now finished if $S$ contains a list for a $4$-part outside $Z_4$, so we
may assume $|S|\le k_3+|Z_4|$.

If $S$ contains the list for a $3$-part $A$ outside $Z_3$, then
$|L_A|\ge k_3+{\FR{k_1+k_4}3}$, by Lemma~\ref{vp3}.  Since $A$ is a $3$-part,
$k_3\ge1$, so ${\FR{k_1+k_4}3}\ge\max\{0,\FR{k_1-k_3+k_4+1}3\} = |Z_4|$, and
$|L_A|\ge |S|$.

Thus we may assume that $S$ contains lists only for parts in $Z_3\cup Z_4$,
and at most one list for each such part.  These are lists for vertices merged
in Section~\ref{greed}.  Within $Z_3$, we performed such merges only for parts
in $Z_3'$, so
$$
|S|\leq |Z_3'|+|Z_4| =
t_3-\left\lceil\frac{k_3}{3}\right\rceil+\max\left\{0,\frac{k_1-k_3+k_4+1}{3}\right\}.
$$

By Construction~\ref{z3}, the list for any merged vertex from a part in $Z_3'$
has size at least $\CL{\FR{k+t_3-1}3}$.  Whether $Z_4$ is empty or not,
$t_3\le k_3$ yields $\CL{\FR{k+t_3-1}3}\ge |Z_3'|+|Z_4|$.

Hence $S$ contains lists only for at most one merged vertex from each part in
$Z_4$.  That is, $|S|\leq |Z_4| =\max\left\{0,\frac{k_4+k_1-k_3+1}3\right\}$.
By Remark~\ref{z4rem}, each such list has size at least ${\FR k3}$.  Since
$k\ge k_1-k_3+k_4$ and $\FR{k_1-k_3+k_4+1}3$ is an integer (by
Corollary~\ref{arith}(a)), always the size of the union of the lists in
$S$ is at least $|S|$.
\end{proof}

\begin{ack}
We thank Bruce Reed for pointing out a simplification of the proof, and we
thank an anonymous referee for many useful suggestions concerning the 
exposition.
\end{ack}


\end{document}